\documentclass[12pt, reqno]{amsart}
\usepackage[left=1in,right=1in,top=1in,bottom=1in]{geometry}
\usepackage{latexsym,amsmath,enumerate,amssymb,epsf,graphicx, comment,appendix,amsthm}
\usepackage{mathrsfs}
\usepackage{tikz}
\usetikzlibrary{arrows.meta}
\usepackage{todonotes}
\usepackage{cite}

\usepackage{mathtools}

\usepackage{hyperref}
\hypersetup{colorlinks=true, linkcolor=blue, citecolor=blue}

\newcommand{\N}{\mathbb{N}}
\newcommand{\Z}{\mathbb{Z}}

\newcommand{\R}{\mathbb{R}}
\newcommand{\F}{\mathbb{F}}
\newcommand{\Pg}{\mathcal{P}}

\theoremstyle{plain}\newtheorem{theorem}{Theorem}[section]
\theoremstyle{plain}\newtheorem{lemma}[theorem]{Lemma}
\theoremstyle{plain}\newtheorem{proposition}[theorem]{Proposition}
\theoremstyle{plain}\newtheorem{corollary}[theorem]{Corollary}
\theoremstyle{plain}

\theoremstyle{remark}
\newtheorem{remark}[theorem]{Remark}

\theoremstyle{definition}
\newtheorem{definition}[theorem]{Definition}

\numberwithin{equation}{section}

\title[Condensed Ricci Curvature]{Condensed Ricci Curvature on Paley Graphs and Their Generalizations}

\author[Bonini]{Vincent Bonini}

\author[Chamberlin]{Daniel Chamberlin}

\author[Cook]{Stephen Cook}

\author[Seetharaman]{Parthiv Seetharaman}

\author[Tran]{Tri Tran}

\begin{document}

\subjclass[2020]{primary 52C99, 53B99; secondary 05C10, 05C81, 05C99}
\keywords{coarse Ricci curvature, Paley graphs, complete subgraphs}

\maketitle


\begin{abstract}
We explore properties of generalized Paley graphs and we extend a result of Lim and Praeger \cite{LimP} by providing a more precise description of the connected components of disconnected generalized Paley graphs. This result leads to a new characterization of when generalized Paley graphs are disconnected. We also provide necessary and sufficient divisibility conditions for the multiplicative group of the prime subfield of certain finite fields to be contained in the multiplicative subgroup of nonzero $k$-th powers. This latter result plays a crucial role in our development of a sorting algorithm on generalized Paley graphs that exploits the vector space structure of finite fields to partition certain subsets of vertices  in a manner that decomposes the induced bipartite subgraph between them into complete balanced bipartite subgraphs. As a consequence, we establish a matching condition between these subsets of vertices that results in an explicit formula for the condensed Ricci curvature on certain Paley graphs and their generalizations. 
\end{abstract}


\section{Introduction}\label{Sec:Intro}
Ollivier defined the coarse Ricci curvature of Markov chains on metric spaces in terms of the Wasserstein (or transport) distance of measures in \cite{Ollivier,Ollivier2}, providing a synthetic notion of Ricci curvature and a bridge between Riemannian geometry and probabilistic methods. The investigation of Ollivier's coarse Ricci curvature on graphs presents an interesting avenue for research and offers an accessible framework for quantifying local connectivity of graphs. Consequently, it has many practical and computational applications in artificial intelligence, network analysis, and data science (cf. \cite{BCDDT, CNS, CLPWZZ, GLLN, GRS, GGLN, HNNNNN, LBL}).

Much work has been done on various forms of the coarse Ricci curvature on graphs (cf. \cite{BauerJostLiu, B1, bakryemery, CKKLMP, LLLY, Liu2024,JostLiu,LLY1,LLY2,Smith,BhatMuk,BCLMP,MR}). We consider a modified notion of Ollivier's coarse Ricci curvature on graphs introduced by Lin, Lu, and Yau in \cite{LLY1} that we refer to as the condensed Ricci curvature as in \cite{B1}. In particular, we study the condensed Ricci curvature on Paley graphs and their generalizations, which serve as models of pseudo-random graphs that encode algebraic relations between the elements of certain finite fields. Standard (or quadratic) Paley graphs $\Pg(q,2)$ are constructed by taking the elements of finite fields of prime power order $q=p^n \equiv 1 \bmod 4$ as vertices and defining edges between vertices that differ by squares. Generalized Paley graphs $\Pg(q,k)$ are constructed in a similar manner by taking the elements of finite fields of prime power order $q=p^n \equiv 1 \bmod 2k$ as vertices and defining edges between vertices that differ by higher order $k$-th powers (see Section \ref{Sec:GenPaleyGraphs}). The connections of Paley graphs and their generalizations to number theory, field theory, and other branches of mathematics adds to their mathematical interest and allows one to apply tools from number theory and algebra in their study.

Generalized Paley graphs $\Pg(q,k)$ share some of the well-known properties of quadratic Paley graphs. Indeed, they are symmetric and $\frac{q-1}{k}$-regular but unlike quadratic Paley graphs they are not self-complementary and may be disconnected when $k >2$ (see Section \ref{Sec:GenPaleyGraphs}). Our main results concerning the condensed Ricci curvature on generalized Paley graphs requires an understanding of their connectivity properties. In a study of the automorphism groups of generalized Paley graphs, it was shown in Theorem 2.2 of \cite{LimP} that a generalized Paley graph $\Pg(q,k)$ of order $q=p^n$ is connected if and only if $k$ is not a multiple of $(q-1)/(p^a-1)$ for any proper divisor $a$ of $n$. Moreover, if $\Pg(q,k)$ is disconnected, then each connected component is isomorphic to the generalized Paley graph $\Pg(p^a,k')$ where $k'=k(p^a-1)/(q-1) \geq 1$ and $a$ is some proper divisor of $n$ such that $(q-1)/(p^a-1)$ divides $k$.
 
From the divisibility conditions in this result of \cite{LimP}, one finds that a generalized Paley graph $\mathcal{P}(q, k)$ is connected if and only if no proper subfield of $\F_q$ contains the multiplicative subgroup of nonzero $k$-th powers. Hence, although it is not explicitly stated in \cite{LimP}, it is natural to expect that each of the connected components of a disconnected generalized Paley graph $\mathcal{P}(q, k)$ is isomorphic to a generalized Paley graph defined over the smallest subfield of $\F_q$ that contains the subgroup of nonzero $k$-th powers. Using basic properties of finite fields and finite cyclic groups, we provide a modest extension of this result of \cite{LimP} by showing that each connected component of a disconnected generalized Paley graph is isomorphic to the generalized Paley graph $\Pg(p^a,k')$ as described above, where $a$ is in fact the smallest such proper divisor of $n$, or equivalently, where $\F_{p^a}$ is the smallest subfield of $\F_q$ containing the subgroup of nonzero $k$-th powers.  

\begin{theorem}\label{Thm:IntroLimPExtension}
Let $k \geq 2$ and suppose that the generalized Paley graph $\mathcal{P}(q, k)$ of order $q=p^n$ is disconnected. Then each connected component of $\mathcal{P}(q, k)$ is isomorphic to the generalized Paley graph $\mathcal{P}(p^a, k')$ with $k' = k (p^a-1)/(q-1) \geq 1$ where $a$ is the smallest proper divisor of $n$ such that $(q-1)/(p^a-1)$ divides $k$.
\end{theorem}

Now suppose that $\theta$ is a primitive element of the finite field $\F_q$ of order $q=p^n$ and let $(\F_q^{\times})^k$ denote the multiplicative subgroup of nonzero $k$-th powers in $\F_q^{\times}$. Noting that the smallest subfield of $\F_q$ that contains $(\F_q^{\times})^k$ is given by the field extension $\F_p(\theta^k)$, it follows that the parameter $a$ in the results of \cite{LimP} and Theorem \ref{Thm:IntroLimPExtension} is precisely the degree of this extension, or equivalently, the degree of the minimal polynomial of $\theta^k$ over $\F_p$. As a result of these observations, we have the following reformulation of the findings of \cite{LimP} and Theorem \ref{Thm:IntroLimPExtension}, which provides a new characterization of when generalized Paley graphs are disconnected.

\begin{theorem}\label{Thm:IntroLimPExtension2}
Let $k \geq 2$ and suppose that $\theta$ is a primitive element of the finite field $\F_q$ of order $q=p^n$ with $q \equiv 1 \bmod 2k$. Then the generalized Paley graph $\mathcal{P}(q, k)$ is disconnected if and only if the field extension $\F_p(\theta^k)$ is a proper subfield of $\F_q$. Furthermore, if $\mathcal{P}(q, k)$ is disconnected, then each connected component is isomorphic to the generalized Paley graph $\mathcal{P}(p^a, k')$ with  
$$k' = \frac{\vert \F_p(\theta^k)^{\times} \vert}{\vert (\F_q^{\times})^k \vert} \geq 1$$
where $a < n$ is the degree of the extension $\F_p(\theta^k)$ over $\F_p$.
\end{theorem}

In light of the aforementioned results of \cite{LimP}, we also provide some cases of interest where we can guarantee that generalized Paley graphs are connected.

\begin{theorem}\label{Thm:IntroSimpleConnected}
Let $k \geq 2$ and suppose that $q=p^n$ is a prime power such that $q \equiv 1 \bmod 2k$. If $k < p^{\frac{n}{2}}+1$, then the generalized Paley graph $\Pg(q,k)$ is connected.
\end{theorem}

\noindent
As a consequence of Theorem \ref{Thm:IntroSimpleConnected} and a straightforward calculus argument, we obtain the fact that all generalized Paley graphs $\Pg(q,k)$ of order $q=p^{km}$ are connected. 

\begin{theorem}\label{Thm:IntrokmConnected}
Let $m\geq 1$, $k \geq 2$ and suppose that $q=p^{km}$ is a prime power such that $q \equiv 1 \bmod 2k$. Then the generalized Paley graph $\mathcal{P}(q, k)$ is connected.  
\end{theorem} 

We then turn our attention to the derivation of an explicit formula for the condensed Ricci curvature on Paley graphs and their generalizations. Given a connected, locally finite, undirected, simple graph $G = (V, E)$ with shortest path distance function $\rho:V \times V \to \N \cup \{0\}$ and vertex $v \in V$, let
$$
\Gamma(v) = \{w \in V \mid \rho(v,w)=1 \}  = \{w \in V \mid vw \in E  \}
$$
denote the subset of vertices that are adjacent $v$. Then for any edge $xy \in E$, one can decompose the neighbor sets $\Gamma(x)$ and $\Gamma(y)$ into disjoint unions
$$
\Gamma(x) = N_x \cup \nabla_{xy} \cup \{y\} \quad \text{and} \quad \Gamma(y) = N_y \cup \nabla_{xy} \cup \{x\}
$$
as in \cite{BhatMuk}, where $\nabla_{xy} = \Gamma(x) \cap \Gamma(y)$ denotes the subset of vertices that are adjacent to both $x$ and $y$ and where
$$
N_x = \Gamma(x) \setminus (\nabla_{xy} \cup \{y\}) \quad \text{and} \quad N_y = \Gamma(y) \setminus (\nabla_{xy} \cup \{x\}).
$$
 
Understanding matchings between the neighbor sets $N_x$ and $N_y$ in a graph $G$ is essential to the calculation of the condensed Ricci curvature $\Bbbk (x, y)$ of an edge $xy \in E$. For example, in \cite{B1} it is shown that for any edge $xy \in E$ of a strongly regular graph of degree $d$, the condensed Ricci curvature
\begin{equation}\label{Eqn:RicciSRG}
\Bbbk (x, y) = \frac{1}{d}(2 + \vert \nabla_{xy} \vert - (|N_x| - m))
\end{equation}
where $m$ is the size of a maximum matching $\mathcal{M}$ between $N_x$ and $N_y$. In general, when a perfect matching exists between the neighbor sets $N_x$ and $N_y$ for every edge $xy \in E$, the graph $G$ is said to satisfying the Global Matching Condition \cite{Smith}. An explicit formula for the condensed Ricci curvature along edges in graphs that satisfy the Global Matching Condition is established in \cite{Smith}.  We state a combined version of Lemma 6.2 and Theorem 6.3 of \cite{Smith} below.

\begin{theorem}[\!\! \cite{Smith}]\label{Thm:IntroSmith} 
Let $G=(V,E)$ be a connected, locally finite, undirected, simple graph satisfying the Global Matching Condition. Then $G$ is regular of degree $d$ and the condensed Ricci curvature
$$\Bbbk(x,y) = \frac{1}{d}(2 + \vert \nabla_{xy} \vert)$$
for any edge $xy \in E$. 
\end{theorem}

Our derivation of an explicit formula for the condensed Ricci curvature on generalized Paley graphs relies on establishing the Global Matching Condition and Theorem \ref{Thm:IntroSmith}. We refer readers to Section 2 of \cite{B1} for the precise definition of the condensed Ricci curvature on graphs and our choice of terminology. The main ideas we use to establish the Global Matching Condition were first conceived for the special case of quadratic Paley graphs $\Pg(q,2)$ of even power order $q=p^{2m}$ in an unpublished student research project \cite{B2}. In particular, in \cite{B2} the fact that the multiplicative group of the corresponding prime subfield $\F_p^{\times}$ is contained in the subgroup of nonzero squares $(\F_q^{\times})^2$ led to the idea that the vector space structure of the finite field $\F_q$ could be used to partition (or ``sort'') the vertices in $N_0$ and $N_1$ in a way that a perfect matching could be easily obtained. Then, by the symmetry of quadratic Paley graphs, there is a perfect matching between the neighbor sets $N_x$ and $N_y$ for every edge $xy \in E$.

We found that the observations of \cite{B2} hold more broadly on certain generalized Paley graphs and that the induced bipartite subgraph of edges between vertices in $N_x$ and $N_y$ can actually be decomposed into complete balanced bipartite subgraphs with a ``sorting algorithm'' (see Section \ref{Sec:GlobalMatchingCond}). With the aim of generalizing the ideas of \cite{B2}, we first establish a substantial case where the multiplicative group of the prime subfield $\F_p^{\times}$ of a finite field $\F_q$ is contained in the subgroup of nonzero $k$-th powers $(\F_q^{\times})^k$.

\begin{theorem}\label{Thm:IntroPrimeBaseField}
Suppose $p$ and $k$ are prime. Then $\F_p^{\times} \leq (\F_{p^{km}}^{\times})^k$ for any positive integer $m$.
\end{theorem}

\noindent
Moreover, as a simple consequence of the properties of finite cyclic groups, we have the following necessary and sufficient divisibility conditions for the multiplicative group of the prime subfield to be contained in the subgroup of nonzero $k$-th powers for the finite fields that serve as vertex sets of generalized Paley graphs.

\begin{theorem}\label{Thm:IntroEquivBaseField}
Let $k \geq 2$ and suppose that $q=p^n$ is a prime power such that $q \equiv 1 \bmod 2k$. Then $\F_p^{\times} \leq (\F_q^{\times})^k$ if and only if $k \mid \frac{q-1}{p-1}$.
\end{theorem}

\noindent
Furthermore, as a straightforward consequence of Theorem \ref{Thm:IntroEquivBaseField}, we find that if $\F_q$ is a finite field of order $q=p^{n} \equiv 1 \bmod 2k$ and $k \mid p-1$, then the multiplicative group of the prime subfield is contained in the subgroup of nonzero $k$-th powers if and only if $n$ is a multiple of $k$. 

\begin{corollary}\label{Cor:IntroknPower}
Let $k \geq 2$ and suppose that $q=p^n$ is a prime power such that $q \equiv 1 \bmod 2k$. If $k \mid p-1$, then $\F_p^{\times} \leq (\F_q^{\times})^k$ if and only if $n \equiv 0 \bmod k$.
\end{corollary}

With Theorems \ref{Thm:IntroPrimeBaseField} and \ref{Thm:IntroEquivBaseField} in hand, we then formally develop the sorting algorithm conceived in \cite{B2} and use it to establish the Global Matching Condition on connected generalized Paley graphs in which the multiplicative group of the corresponding prime subfield is contained in the subgroup of nonzero $k$-th powers (see Section \ref{Sec:GlobalMatchingCond}).

\begin{theorem}\label{Thm:IntroGlobalMatch}
Let $k \geq 2$ and suppose $\Pg(q,k)$ is a connected generalized Paley graph of order $q=p^n$. If $\F_p^{\times} \leq (\F_q^{\times})^k$, then $\Pg(q,k)$ satisfies the Global Matching Condition.
\end{theorem}

Due to Theorems \ref{Thm:IntrokmConnected}, \ref{Thm:IntroSmith}, \ref{Thm:IntroPrimeBaseField}, and \ref{Thm:IntroGlobalMatch}, we then obtain the following generalization of the unpublished work of \cite{B2} on quadratic Paley graphs $\Pg(q,2)$ of even power order $q=p^{2m}$.

\begin{theorem}\label{Thm:PrimeRicci}
Let $m \geq 1$ and suppose that $\Pg(q,k)=(V,E)$ is a generalized Paley graph of order $q=p^{km}$ where $k$ is prime. Then the condensed Ricci curvature
$$\Bbbk(x,y) = \frac{k}{q-1}(2 + \vert \nabla_{xy} \vert)$$
for any edge $xy \in E$.  
\end{theorem}

\noindent
More generally, as a consequence of Theorems \ref{Thm:IntroSmith}, \ref{Thm:IntroEquivBaseField}, and \ref{Thm:IntroGlobalMatch}, we obtain the same explicit formula for the condensed Ricci curvature on generalized Paley graphs $\Pg(q,k)$ of order $q=p^n$ satisfying the divisibility condition $k \mid \frac{q-1}{p-1}$.

\begin{theorem}\label{Thm:IntroRicci}
Let $k \geq 2$ and suppose that $\Pg(q,k)=(V,E)$ is a generalized Paley graph of order $q=p^n$ where $k \mid \frac{q-1}{p-1}$. Then the condensed Ricci curvature
$$\Bbbk(x,y) = \frac{k}{q-1}(2 + \vert \nabla_{xy} \vert)$$
for any edge $xy \in E$.  
\end{theorem}

We would like to point out that the formula for the condensed Ricci curvature in Theorem \ref{Thm:IntroRicci} holds for both connected and disconnected generalized Paley graphs. However, in our work we consider these cases separately (see Theorems \ref{Thm:Ricci} and \ref{Thm:DisconnectedRicci}) and we recover the same formula for the condensed Ricci curvature by applying our results for connected graphs to the connected components of disconnected generalized Paley graphs. In the special case that the connected components of a disconnected generalized Paley graph are complete graphs, we appeal to the following result stated in \cite{LLY1} and proved in \cite{B1}.

\begin{theorem}[\!\! \cite{B1, LLY1}]\label{Thm:CompleteGraphsIntro}
A connected, finite, undirected, simple graph $G = (V, E)$ is complete if and only if the condensed Ricci curvature $\Bbbk(x,y) > 1$ for all vertices $x,y \in V$. In particular, if $G$ is a complete graph on $n$ vertices, then $\Bbbk(x,y) =\frac{n}{n-1}$ for all vertices $x,y \in V$. 
\end{theorem}

\begin{remark} Generalized Paley graphs $\Pg(q,k)$ can be defined over finite fields of even order $q=2^n$ as in \cite{LimP}. In this case, one requires $q \equiv 1 \bmod k$ rather than requiring $q \equiv 1 \bmod 2k$ as in the case for odd prime powers. Our work applies to generalized Paley graphs of even order, although some our results are trivial when considering finite fields of characteristic $2$. However, for simplicity in presentation, we have chosen to focus on generalized Paley graphs of odd order.
\end{remark}

This paper is organized as follows: In Section \ref{Sec:GenPaleyGraphs} we formally define generalized Paley graphs and we discuss some of their relavant properties. In particular, we extend the results \cite{LimP} by providing a more precise description of the connected components of disconnected generalized Paley graphs and we give a new characterization of when generalized Paley graphs are disconnected. In Section \ref{Sec:GlobalMatchingCond} we establish results on the containment of the multiplicative group of the prime subfield of a finite field in the subgroup of nonzero $k$-th powers and we develop a ``sorting algorithm'', which leads to establishing the Global Matching Condition and our explicit formulas for the condensed Ricci curvature on Paley graphs and their generalizations.


\section*{Acknowledgements}
This research was generously supported by the William and Linda Frost Fund in the Cal Poly Bailey College of Science and Mathematics. The authors would like to thank the referee for their careful reading of our work and their valuable input and insight. We also extend our gratitude to Professor Eric Brussel and Professor Rob Easton of Cal Poly for many vauble conversations.


\section{Generalized Paley Graphs and Their Properties}\label{Sec:GenPaleyGraphs}

In this section we introduce the definition of generalized Paley graphs and we discuss some of their relevant properties. Consider the finite field $\F_q$ of order $q=p^n$ and let $\F_q^{\times}$ denote the multiplicative group of $\F_q$. Then the set of nonzero squares in $\F_q$ is denoted by
$$(\F_q^{\times})^2 = \{\alpha \in \F_q^{\times} \mid \alpha = \beta^2  \ \text{for some} \ \beta \in \F_q^{\times}\}.$$
Paley graphs are then constructed by taking the field elements of certain finite fields as vertices and defining edges between those vertices that differ by squares.

\begin{definition}\label{Def:QPaleyGraphs}
Let $q=p^n$ be a prime power such that $q \equiv 1 \bmod 4$. Then the {\it Paley graph of order} $q$ is defined to be the graph $\mathcal{P}(q)=(V,E)$ with vertex set $V=\F_q$ and edge set $E=\{xy \mid x-y \in (\F_q^{\times})^2\}$.
\end{definition}

We refer to Paley graphs with edges between vertices that differ by squares as quadratic Paley graphs. It is known that quadratic Paley graphs are connected, self-complementary, strongly regular graphs with parameters $(q, \frac{q-1}{2}, \frac{q-5}{4}, \frac{q-1}{4})$ (cf. \cite{Jones}). Naturally, one can generalize the definition of Paley graphs to define edges in terms of higher order powers. For integers $k \geq 2$, we denote the set of nonzero $k$-th powers in $\F_q$ by
\begin{equation}\label{Eqn:kResidues}
(\F_q^{\times})^k = \{\alpha \in \F_q^{\times} \mid \alpha = \beta^k  \ \text{for some} \ \beta \in \F_q^{\times}\}.
\end{equation}

\begin{definition}\label{Def:kPaleyGraphs}
Let $k \geq 2$ and suppose $q=p^n$ is a prime power such that $q \equiv 1 \bmod 2k$. Then the {\it generalized Paley graph of order} $q$ {\it with} $k$-{\it th powers} is defined to be the graph $\mathcal{P}(q, k)=(V,E)$ with vertex set $V=\F_q$ and edge set $E=\{xy \mid x-y \in (\F_q^{\times})^k\}$.
\end{definition}

Note that one can consider $k=1$ in \eqref{Eqn:kResidues} and Definition \ref{Def:kPaleyGraphs}. In this case the set of nonzero $1$-st powers in $\F_q$ is simply the multiplicative group $\F_q^{\times}$ and therefore $\Pg(q,1)=K_{q}$ is the complete graph on $q$ vertices. For simplicity in our presentation we will sometimes refer to complete graphs as generalized Paley graphs. For $k=2$, the generalized Paley graphs $\mathcal{P}(q,2)$ are quadratic Paley graphs as defined in Definition \ref{Def:QPaleyGraphs}. It is natural to refer to generalized Paley graphs with $k=3$ as cubic, $k=4$ as quartic, $k=5$ as quintic, and so on. Throughout this work we refer to Paley graphs with edges between vertices that differ by $k$-th powers simply as generalized Paley graphs or $k$-Paley graphs. 

We also note that the set of nonzero $k$-th powers of a finite field $\F_q$ form a subgroup of $\F_q^{\times}$. Moreover, if $\theta$ is a generator for the multiplicative cyclic group $\F_q^{\times}$, then $\theta^k$ generates the subgroup $(\F_q^{\times})^k$ of $k$-th powers. Hence, since $\vert \theta \vert = q-1$ and we are considering finite fields $\F_q$ where $k \mid q-1$ , it follows that $\vert \theta^k \vert = \frac{q-1}{k}$ and therefore $\frac{q-1}{k}$ elements of $\F_q^{\times}$ are $k$-th powers, or equivalently, the subgroup of nonzero $k$-th powers in $\F_q^{\times}$ has order $\frac{q-1}{k}$.
 
The condition $q \equiv 1 \bmod 2k$ in Definition \ref{Def:kPaleyGraphs} ensures that $q-1$ is even, so $q=p^n$ must be an odd prime power for the generalized Paley graphs under consideration. It also guarantees that generalized Paley graphs are undirected. In other words, the condition $2k \mid q-1$ guarantees that if $x-y \in (\F_q^{\times})^k$ for some $x, y \in \F_q$, then $y-x \in (\F_q^{\times})^k$, or equivalently that $-1 \in (\F_q^{\times})^k$. Indeed, if $\theta$ is a generator for the multiplicative group $\F_q^{\times}$, then
 $$1=(\theta^{\frac{q-1}{2}})^2$$
and therefore $\theta^{\frac{q-1}{2}}=-1$. 
But then since $\frac{q-1}{2k} \in \Z$, it follows that $-1 = (\theta^{\frac{q-1}{2k}})^k \in(\F_q^{\times})^k$. 

Generalized Paley graphs retain some of the basic properties shared by quadratic Paley graphs. For example, generalized Paley graphs $\Pg(q,k)$ are symmetric as they are easily seen to be arc-transitive under the subgroup of affine automorphims of the form $x \mapsto ax+b$ where $a \in (\F_q^{\times})^k$ and $b \in \F_q$. Moreover, $k$-Paley graphs are $\frac{q-1}{k}$-regular. For completeness, we record and prove these results in the following propositions.

\begin{proposition}\label{Prop:Symmetry}
Generalized Paley graphs $\mathcal{P}(q, k) = (V,E)$ of order $q=p^n$ are symmetric. 
\end{proposition}

\begin{proof}
For any edges $x_1y_1, x_2y_2 \in E$, taking $$a=(y_2-x_2)(y_1-x_1)^{-1} \in (\F_q^{\times})^k \ \ \text{and} \ \ b = x_2 - a x_1 \in \F_q^{\times},$$
it follows that the automorphism $\phi:V \to V$ defined by $\phi(x) = ax+b$ satisfies $\phi(x_1) = x_2$ and $\phi(y_1) = y_2.$ Thus, $\Pg(q,k)$ is arc-transitive and therefore symmetric.
\end{proof}

\begin{proposition}\label{Prop:Regular}
Generalized Paley graphs $\mathcal{P}(q, k) = (V,E)$ of order $q=p^n$ are $\frac{q-1}{k}$-regular. 
\end{proposition}

\begin{proof}
Let $\theta$ be a generator for $\F_q^{\times}$ and suppose $x \in V$. Then for any other vertex $y \in V$, it follows that $y \in \Gamma(x)$ if and only if $y-x \in (\F_q^{\times})^k$. Therefore, since $\theta^k$ generates the subgroup of nonzero $k$-th powers, it follows that $y \in \Gamma(x)$ if and only if $y = x + (\theta^k)^m$ for some $m \in \Z$. Thus, since $\vert \theta^k \vert = \frac{q-1}{k}$, we see that $\vert \Gamma(x) \vert = \frac{q-1}{k}$. Hence, since $x$ was chosen arbitrarily, it follows that $\mathcal{P}(q, k)$ is $\frac{q-1}{k}$-regular.   
\end{proof}

In contrast to the case for quadratic Paley graphs, generalized Paley graphs $\Pg(q,k)$ are not self-complementary for $k > 2$. This follows directly from the fact that self-complementary graphs with $q$ vertices must have $\frac{q(q-1)}{4}$ edges, that is, half the number of edges as in a complete graph on $q$ vertices. But $\Pg(q,k)$ is $\frac{q-1}{k}$-regular by Proposition \ref{Prop:Regular} and therefore has
$$\frac{1}{2} \cdot q \cdot \frac{q-1}{k} = \frac{q(q-1)}{2k}$$
edges, which is less than $\frac{q(q-1)}{4}$ for $k>2$. Furthermore, generalized Paley graphs $\Pg(q,k)$ may not be connected when $k > 2$. However, even when a generalized Paley graph is disconnected, it turns out that each of its connected components is isomorphic to a single generalized Paley graph defined over a proper subfield of $\F_q$. These latter properties are due to Theorem 2.2 of \cite{LimP}. We restate the relevant portions of this theorem in the context of our work below.

\begin{theorem}[Theorem 2.2 \cite{LimP}]\label{Thm:LimP}
For $k \geq 2$, the generalized Paley graph $\mathcal{P}(q, k)$ of order $q=p^n$ is connected if and only if $k$ is not a multiple of $(q-1)/(p^a-1)$ for any proper divisor $a$ of $n$. Furthermore, if $\mathcal{P}(q, k)$ is disconnected, then each connected component is isomorphic to the generalized Paley graph $\mathcal{P}(p^a, k')$ with $k' = k (p^a-1)/(q-1) \geq 1$ where $a$ is some proper divisor of $n$ such that $(q-1)/(p^a-1)$ divides $k$.
\end{theorem} 

As noted in the introduction, the divisibility conditions in Theorem \ref{Thm:LimP} imply that a generalized Paley graph $\mathcal{P}(q, k)$ is connected if and only if no proper subfield of $\F_q$ contains the subgroup of nonzero $k$-th powers. Indeed, since $k \mid q-1$, it follows that $(q-1)/(p^a-1)$ divides $k$ if and only if $\vert (\F_q^{\times})^k \vert = \frac{q-1}{k}$ divides $\vert \F_{p^a}^{\times} \vert = p^a-1$, or equivalently, $(\F_q^{\times})^{k} \leq \F_{p^a}^{_{\times}}$. Therefore, it is natural to expect that each of the connected components of a disconnected generalized Paley graph $\mathcal{P}(q, k)$ is isomorphic to a generalized Paley graph defined over the smallest subfield of $\F_q$ that contains the subgroup of nonzero $k$-th powers. 

Now if $q=p^n$ is a prime power such that $q \equiv 1 \bmod 2k$ and $b$ is a proper divisor of $n$ for which $(q-1)/(p^b-1)$ divides $k$, then $k = k'(q-1)/(p^b-1)$ for some $k' \in \Z$ so 
\begin{equation}\label{Eqn:k'Divides}
p^b-1= k' \frac{q-1}{k} =2k' \frac{q-1}{2k}.
\end{equation}
But $\frac{q-1}{2k} \in \Z$ so by \eqref{Eqn:k'Divides} it follows that $2k' \mid p^b-1$ and therefore $p^b \equiv 1 \bmod 2k'$. Thus, one obtains a generalized Paley graph $\Pg(p^b,k')$ over the proper subfield $\F_{p^b}$ of $\F_q$ for every proper divisor $b$ of $n$ for which $(q-1)/(p^b-1)$ divides $k$. However, if $b$ is not the smallest such divisor of $n$, that is if $\F_{p^b}$ is not the smallest subfield of $\F_q$ containing the subgroup of nonzero $k$-th powers, then it turns out that $\Pg(p^b,k')$ is also disconnected.

Although it is not explicitly stated in \cite{LimP}, a further analysis shows that each of the connected components of a disconnected generalized Paley graph $\Pg(q,k)$ is actually isomorphic to the generalized Paley graph $\mathcal{P}(p^a, k')$ as described in Theorem \ref{Thm:LimP} where $a$ is in fact the smallest proper divisor of $n$ such that $(q-1)/(p^a-1)$ divides $k$, or equivalently, where $\F_{p^a}$ is the smallest subfield of $\F_q$ containing the subgroup of nonzero $k$-th powers.

\begin{theorem}\label{Thm:LimPExtension}
Let $k \geq 2$ and suppose the generalized Paley graph $\mathcal{P}(q, k)$ of order $q=p^n$ is disconnected. Then each connected component of $\mathcal{P}(q, k)$ is isomorphic to the generalized Paley graph $\mathcal{P}(p^a, k')$ with $k' = k (p^a-1)/(q-1) \geq 1$ where $a$ is the smallest proper divisor of $n$ such that $(q-1)/(p^a-1)$ divides $k$.
\end{theorem}

\begin{proof}
By Theorem \ref{Thm:LimP} we may assume that $(q-1)/(p^b-1)$ divides $k$ for some proper divisor $b$ of $n$. Equivalently, since $k \mid q-1$, we may write this divisibility condition as $\frac{q-1}{k} \mid p^b-1$. Therefore, since $\F_{p^b}^{\times}$ is cyclic and $\vert (\F_q^{\times})^k \vert = \frac{q-1}{k}$ divides $\vert \F_{p^b}^{\times} \vert = p^b-1$, it follows that
\begin{equation}\label{Eqn:kResContainment} 
(\F_q^{\times})^k \leq \F_{p^{b}}^{\times}.
\end{equation}

Let $a$ be the smallest proper divisor of $n$ such that $(q-1)/(p^a-1)$ divides $k$. If $a < b$, it follows that $(\F_q^{\times})^k \leq \F_{p^a}^{\times}$ as in \eqref{Eqn:kResContainment} and therefore
\begin{equation}\label{Eqn:gcdContainment} 
(\F_q^{\times})^k \leq \F_{p^a}^{\times} \cap \F_{p^b}^{\times} = \F_{p^d}^{\times}
\end{equation}
where $d=\operatorname{gcd}(a,b)$. Thus, by Lagrange's theorem, it follows from \eqref{Eqn:gcdContainment} that $\frac{q-1}{k} \mid p^d-1$, or equivalently
$(q-1)/(p^d-1)$ divides $k$ where $d=\operatorname{gcd}(a,b) \leq a$ is a proper divisor of $n$. But since $a$ is the smallest proper divisor of $n$ such that $(q-1)/(p^a-1)$ divides $k$, it follows that $a=d$ and therefore $a \mid b$.

Now let $k'_a = k(p^a-1)/(q-1)$ and $k'_b = k(p^b-1)/(q-1)$. Then
$$k'_b= k\frac{p^b-1}{q-1}=k'_a \frac{q-1}{p^a-1} \frac{p^b-1}{q-1} = k'_a \frac{p^b-1}{p^a-1}$$
and therefore $(p^b-1)/(p^a-1)$ divides $k'_b.$ Hence, since $a$ is a proper divisor of $b$, it follows from Theorem \ref{Thm:LimP} that the generalized Paley graph $\Pg(p^b,k'_b)$ is disconnected. Thus, in light of Theorem \ref{Thm:LimP}, the connected components of a disconnected generalized Paley graph $\Pg(q,k)$ must be isomorphic to the generalized Paley graph $\Pg(p^a,k')$ with $k' = k (p^a-1)/(q-1) \geq 1$ where $a$ is the smallest proper divisor of $n$ such that $(q-1)/(p^a-1)$ divides $k$, or equivalently, where $\F_{p^a}$ is the smallest subfield of $\F_q$ containing the subgroup of nonzero $k$-th powers.
\end{proof}

Now suppose that $\theta$ is a primitive element of the finite field $\F_q$ of order $q=p^n$ with $q \equiv 1 \bmod 2k$. Then the smallest subfield of $\F_q$ that contains the subgroup of nonzero $k$-th powers is given by the field extension $\F_p(\theta^k)$ over $\F_p$. Thus, if the degree of this extension is strictly less than $n$, say
$$[\F_p(\theta^k):\F_p] = a < n,$$
it follows that $\F_p(\theta^k) = \F_{p^a}$ is a proper subfield of $\F_q$ where $a$ is the smallest proper divisor of $n$ such that $\vert (\F_q^{\times})^k \vert = \frac{q-1}{k}$ divides $\vert \F_p(\theta^k)^{\times} \vert = p^a-1$, or equivalently, such that $(q-1)/(p^a-1)$ divides $k$. As a result of these observations, we see that the parameter $a$ in Theorem 2.5 of \cite{LimP} and Theorem 2.6 is precisely the degree of the field extension $\F_p(\theta^k)$ over $\F_p$, or equivalently, the degree of the minimal polynomial of $\theta^k$ over $\F_p$. Therefore, we have the following reformulation of the results of \cite{LimP} and Theorem 2.6.

\begin{theorem}\label{Thm:LimPExtension2}
Let $k \geq 2$ and suppose that $\theta$ is a primitive element of the finite field $\F_q$ of order $q=p^n$ with $q \equiv 1 \bmod 2k$. Then the generalized Paley graph $\mathcal{P}(q, k)$ is disconnected if and only if the field extension $\F_p(\theta^k)$ is a proper subfield of $\F_q$. Furthermore, if $\mathcal{P}(q, k)$ is disconnected, then each connected component is isomorphic to the generalized Paley graph $\mathcal{P}(p^a, k')$ with
$$k' = \frac{\vert \F_p(\theta^k)^{\times} \vert}{\vert (\F_q^{\times})^k \vert} \geq 1$$
where $a < n$ is the degree of the extension $\F_p(\theta^k)$ over $\F_p$.
\end{theorem}

With these facts in hand, we focus our attention on connected generalized Paley graphs as our results will easily lend themselves to the connected components of disconnected generalized Paley graphs. One simple case where we can guarantee that generalized Paley graphs are connected is given in the following theorem.

\begin{theorem}\label{Thm:SimpleConnected}
Let $k \geq 2$ and suppose that $q=p^n$ is a prime power such that $q \equiv 1 \bmod 2k$. If $k < p^{\frac{n}{2}}+1$, then the generalized Paley graph $\Pg(q,k)$ of order $q=p^n$ is connected.
\end{theorem}

\begin{proof}
Suppose that $a$ is a proper divisor of $n$. Then
$$\frac{p^{n}-1}{p^a-1} \geq \frac{p^{n}-1}{p^{\frac{n}{2}}-1} = p^{\frac{n}{2}}+1 > k.$$
Hence, $k$ is not a multiple of $(p^n-1)/(p^a-1)$ for any proper divisor $a$ of $n$ and therefore $\mathcal{P}(q, k)$ is connected by Theorem \ref{Thm:LimP}.
\end{proof}

To conclude this section we present another special case where the generalized Paley graphs under consideration are connected. In particular, we show that all generalized Paley graphs $\Pg(q,k)$ of order $q=p^{km}$ are connected. This result contributes to a complete generalization of the observations outlined in \cite{B2} for quadratic Paley graphs $\Pg(q,2)$ of even power order $q=p^{2m}$.

\begin{theorem}\label{Thm:kmConnected}
Let $m\geq 1$, $k \geq 2$ and suppose that $q=p^{km}$ is a prime power such that $q \equiv 1 \bmod 2k$. Then the generalized Paley graph $\mathcal{P}(q, k)$ of order $q=p^{km}$ is connected. 
\end{theorem}

\begin{proof}
It suffices to show that $km > 2\log_p(k-1)$ since then $p^{\frac{km}{2}}+1 > k$ and therefore $\mathcal{P}(q, k)$ is connected by Theorem \ref{Thm:SimpleConnected}. To this end, consider the function $f:(1,\infty) \to \R$ defined by
$$f(x) = x - 2\log_2(x-1) = x - \frac{2}{\ln 2}\ln{(x-1)}.$$
Then
$$f'(x) = 1 - \frac{2}{(\ln{2})(x-1)}$$
so $f$ has a single critical point at 
$$x^* = \frac{2}{\ln{2}}+1$$
where $f'<0$ on $(1,x^*)$ and $f'>0$ on $(x^*,\infty)$. Thus, $f$ is decreasing on $(1,x^*)$ and increasing on $(x^*,\infty)$. Moreover, since 
$$f(x^*) = \frac{2}{\ln{2}}\left(1-\ln{\left(\frac{2}{\ln{2}}\right)}\right)+1 > 0,$$
it follows that $f$ has a positive global minimum at $x^*$. Hence, $f > 0$ on $(1,\infty)$ and therefore $x > 2\log_2(x-1)$ for all $x \in (1,\infty)$. In particular, since $k \geq 2$ and $p$ is necessarily an odd prime, we see
\begin{equation}\label{Eqn:Ineq3}
km \geq k > 2\log_2(k-1) > 2\log_p(k-1) 
\end{equation}
and the desired result follows.
\end{proof}


\section{The Sorting Algorithm and The Global Matching Condition}\label{Sec:GlobalMatchingCond}
In this section we establish the Global Matching Condition for certain Paley graphs and their generalizations. As a consequence of Theorem \ref{Thm:IntroSmith} of \cite{Smith} we then obtain an explicit formula for the condensed Ricci curvature along the edges of these graphs. Consider a connected generalized Paley graph $\Pg(q,k)$ of order $q=p^n$ with vertex set $V = \F_q$ and edge set $E = \{xy \mid x-y \in (\F_q^{\times})^k\}$. Due to the symmetry of generalized Paley graphs we may focus on the edge $01 \in E$ and the neighbor sets 
$$N_0 = \Gamma(0) \setminus (\nabla_{01} \cup \{1\}) \quad \text{and} \quad N_1 = \Gamma(1) \setminus (\nabla_{01} \cup \{0\})$$ 
as introduced in Section \ref{Sec:Intro}. 

Let $H$ denote the induced bipartite subgraph consisting of all edges in $E$ between vertices in $N_0$ and $N_1$. Our strategy in establishing the Global Matching Condition on the generalized Paley graphs under consideration is to formally develop and generalize the sorting algorithm that was first conceived in \cite{B2} for the special case of quadratic Paley graphs. We then apply this sorting algorithm to decompose the bipartite subgraph $H$ into complete balanced bipartite subgraphs. Then one can pairwise match vertices in these subgraphs to construct a perfect matching between $N_0$ and $N_1$ and appeal to the symmetry of generalized Paley graphs to realize a perfect matching between the neighbor sets $N_x$ and $N_y$ for every edge $xy \in E$. 

On generalized Paley graphs the containment of the corresponding multiplicative group of the prime subfield in the subgroup of nonzero $k$-th powers plays a crucial role in the success of the sorting algorithm. With the aim of establishing an important case of this critical component of the sorting algorithm, we first present a special case of Theorem 9.1 in chapter $6$, section $9$ of \cite{Lang}.

\begin{theorem}[Theorem 9.1 \cite{Lang}]\label{Thm:LangIrred}
Let $\F_p$ denote the field of prime order $p$ and suppose $k$ is prime. Let $\alpha \in \F_p^{\times}$ and suppose that $\alpha \not\in (\F_p^{\times})^k$. Then $x^k - \alpha$ is irreducible in $\F_p[x].$  
\end{theorem}

For the case of quadratic Paley graphs $\Pg(q,2)$ of even power order $q=p^{2m}$, it was observed in \cite{B2} that the corresponding multiplicative group of the prime subfield $\F_p^{\times}$ is contained in the subgroup of nonzero squares $(\F_{q}^{\times})^2$. We appeal to Theorem \ref{Thm:LangIrred} to establish a more general version of this observation. 

\begin{theorem}\label{Thm:PrimeBaseField}
Suppose $p$ and $k$ are prime. Then $\F_p^{\times} \leq (\F_{p^{km}}^{\times})^k$ for any positive integer $m$.
\end{theorem}

\begin{proof}
Let $\alpha \in \F_p^{\times}$ and suppose that $\alpha \not\in (\F_p^{\times})^k$. Then by Theorem \ref{Thm:LangIrred}, $f(x) = x^k-\alpha$ is irreducible in $\F_p[x]$. Therefore, if $\theta$ is a root of $f$, then
$$\F_p[x] \slash \langle x^k-\alpha \rangle \cong \F_p(\theta) \cong \F_{p^k}.$$ 

\noindent
Now suppose that $\varphi:\F_p(\theta) \to \F_{p^k}$ is an isomorphism. Noting that $\varphi(\alpha)=\alpha$ since $\alpha \in \F_p$, it follows that
$$\varphi(\theta)^k - \alpha = \varphi(\theta^k) - \varphi(\alpha) = \varphi(\theta^k - \alpha) = \varphi(f(\theta)) = \varphi(0) = 0.\vspace{0.1cm}$$
Thus, $\alpha = \varphi(\theta)^k \in (\F_{p^{k}}^{\times})^k$ and therefore since $\F_{p^k} \subseteq \F_{p^{km}}$ for any positive integer $m$, it follows that $\alpha \in (\F_{p^{km}}^{\times})^k$. Hence, $\F_p^{\times} \leq (\F_{p^{km}}^{\times})^k$ as desired. 
\end{proof}

Now suppose $\F_q$ is a finite field of order $q=p^n$ such that $q \equiv 1 \bmod 2k$ with $k \geq 2$. As a simple consequence of the properties of finite cyclic groups, we have the following necessary and sufficient divisibility conditions for the multiplicative group of the prime subfield of $\F_q$ to be contained in the subgroup of nonzero $k$-th powers.

\begin{theorem}\label{Thm:EquivBaseField}
Let $k \geq 2$ and suppose that $q=p^n$ is a prime power such that $q \equiv 1 \bmod 2k$. Then $\F_p^{\times} \leq (\F_q^{\times})^k$ if and only if $k \mid \frac{q-1}{p-1}$.
\end{theorem}

\begin{proof}
Let $\F_q$ be a finite field of order $q=p^n$ such that $q \equiv 1 \bmod 2k$. Then since
$$\vert \F_p^{\times} \vert = p-1 \quad \text{and} \quad \vert (\F_{q}^{\times})^k \vert = \frac{q-1}{k},$$
it follows from properties of finite cyclic groups and Lagrange's theorem that $\F_p^{\times} \leq (\F_q^{\times})^k$ if and only if $p-1 \mid \frac{q-1}{k}$, or equivalently $k \mid \frac{q-1}{p-1}$.
\end{proof}

Clearly, since we are considering generalized Paley graphs $\Pg(q,k)$ of order $q=p^n$ where $q \equiv 1 \bmod 2k$, it follows that $k \mid q-1$. 
Due to Theorem \ref{Thm:EquivBaseField}, it turns out that when $k$ also divides $p-1$, the multiplicative group of the corresponding prime subfield $\F_p^{\times}$ is contained in the subgroup of nonzero $k$-th powers $(\F_q^{\times})^k$ if and only if $n$ is a multiple of $k$.

\begin{corollary}\label{Cor:knPower}
Let $k \geq 2$ and suppose that $q=p^n$ is a prime power such that $q \equiv 1 \bmod 2k$. If $k \mid p-1$, then $\F_p^{\times} \leq (\F_q^{\times})^k$ if and only if $n \equiv 0 \bmod k$.
\end{corollary}

\begin{proof}\label{Cor:nMultk}
By Theorem \ref{Thm:EquivBaseField}, $\F_p^{\times} \leq (\F_q^{\times})^k$ if and only if $k \mid \frac{q-1}{p-1}$.
Hence, $\F_p^{\times} \leq (\F_q^{\times})^k$ if and only if
\begin{equation}\label{Eqn:Mod1}
\frac{q-1}{p-1} = p^{n-1} + p^{n-2} + \cdots + p + 1 \equiv 0 \bmod k.
\end{equation}
But $k \mid p-1$ so $p \equiv 1 \bmod k$ and therefore
\begin{equation}\label{Eqn:Mod2}
\frac{q-1}{p-1} = p^{n-1} + p^{n-2} + \cdots + p + 1 \equiv n \bmod k.
\end{equation}
Hence, by \eqref{Eqn:Mod1} and \eqref{Eqn:Mod2} it follows that $\F_p^{\times} \leq (\F_q^{\times})^k$ if and only if $n \equiv 0 \bmod k$.
\end{proof}

Now that we have a better understanding of when the multiplicative group of the prime subfield is contained in the subgroup of nonzero $k$-th powers of a given finite field, we develop the sorting algorithm that we will use to construct a perfect matching between the neighbor sets $N_0$ and $N_1$ in a generalized Paley graph. To this end, let $\theta$ be a generator of the multiplicative group $\F_q^{\times}$ of the finite field $\F_q$ of order $q=p^n$. Consider the minimal polynomial $f \in \F_p[x]$ of $\theta$ such that 
\begin{equation}\label{Eqn:FieldIsom}
\F_q \cong \F_p(\theta) \cong \F_p[x] \slash \langle f(x) \rangle.
\end{equation}
In light of the isomorphisms in \eqref{Eqn:FieldIsom}, it follows that $1, \theta, \theta^2,\dots,\theta^{n-1}$ is a basis for $\F_q$ as a vector space over $\F_p$ and therefore 
\begin{equation}\label{Eqn:VS}
\F_q = \{ a_0 +  a_1\theta + a_2 \theta^2 + \cdots + a_{n-1}\theta^{n-1} \mid a_0, a_1,\dots,a_{n-1} \in \F_p\}.
\end{equation}

Our formalization of the sorting algorithm and substantiation of a perfect matching between the neighbor sets $N_0$ and $N_1$ relies on the following fact that was first observed for quadratic Paley graphs in \cite{B2}.
 
\begin{lemma}\label{Lem:PartitionSets}
Let $\Pg(q,k)$ be a connected generalized Paley graph of order $q=p^n$ and suppose that $\theta$ is a generator of the multiplicative group $\F_q^{\times}$. Fix $a_1,\dots, a_{n-1} \in \F_p$ not all zero and set
$$S=S(a_1,\dots,a_{n-1})=\{ b + a_1\theta + a_2 \theta^2 + \cdots + a_{n-1}\theta^{n-1} \mid b\in \F_p\}.$$ 
Then $\vert N_0 \cap S \vert = \vert N_1 \cap S \vert$.
\end{lemma}

\begin{proof}
First note that $x \in \Gamma(0)$ if and only if $x \in (\F_q^{\times})^k$ if and only if $x+1 \in \Gamma(1)$ and therefore 
\begin{equation}\label{Eqn:Int1Step}
\vert \Gamma(0) \cap S \vert = \vert \Gamma(1) \cap S \vert.
\end{equation}
Now since $a_1,\dots, a_{n-1}  \in \F_p$ are not all zero, it follows that
$$\{0\} \cap S = \varnothing \quad \text{and} \quad \{1\} \cap S = \varnothing.$$
Moreover, since
$$ \Gamma(0) = \{1\} \cup N_0 \cup \nabla_{01} \quad \text{and} \quad \Gamma(1) = \{0\} \cup N_1 \cup \nabla_{01}$$
are disjoint unions, it follows that
\begin{equation}\label{Eqn:SizeG0}
\vert \Gamma(0) \cap S \vert = \vert N_0 \cap S \vert + \vert \nabla_{01} \cap S \vert
\end{equation}
and
\begin{equation}\label{Eqn:SizeG1}
\vert \Gamma(1) \cap S \vert = \vert N_1 \cap S \vert + \vert \nabla_{01} \cap S \vert.\vspace{0.1cm}
\end{equation}
Thus, by \eqref{Eqn:Int1Step} we may equate equations \eqref{Eqn:SizeG0} and \eqref{Eqn:SizeG1} to find $\vert N_0 \cap S \vert = \vert N_1 \cap S \vert$ . 
\end{proof}

Lemma \ref{Lem:PartitionSets} provides an avenue for partitioning the neighbor sets $N_0$ and $N_1$ of a generalized Paley graph $\Pg(q,k)$ by grouping vertices (or elements) in $N_0$ and $N_1$ with the same coefficients of $\theta,\dots,\theta^{n-1}$. In particular, when the nonzero elements of the prime subfield are all $k$-th powers, we can partition the induced bipartite subgraph $H$ consisting of all edges in $E$ between $N_0$ and $N_1$ into complete balanced bipartite subgraphs. By matching vertices in each of these complete subgraphs in pairwise fashion, we obtain a perfect matching between $N_0$ and $N_1$, which by symmetry shows that the generalized Paley graphs under consideration satisfy the Global Matching Condition. This approach is based on ideas from \cite{B2} where the argument was outlined for quadratic Paley graphs of even power order $q=p^{2m}$.

\begin{theorem}\label{Thm:GlobalMatch}
Let $k \geq 2$ and suppose $\Pg(q,k)$ is a connected generalized Paley graph of order $q=p^n$. If $\F_p^{\times} \leq (\F_q^{\times})^k$, then $\Pg(q,k)$ satisfies the Global Matching Condition.
\end{theorem}

\begin{proof}
By symmetry it suffices to consider the edge $01 \in E$ and to show that there is a perfect matching between $N_0$ and $N_1$.  As in Lemma \ref{Lem:PartitionSets}, fix $a_1,\dots, a_{n-1} \in \F_p$ not all zero and set 
$$S(a_1,\dots,a_{n-1}) = \{ b + a_1\theta + a_2 \theta^2 + \cdots + a_{n-1}\theta^{n-1} \mid b\in \F_p\}.$$ 
Then for any $\alpha \in N_0 \cap S(a_1,\dots,a_{n-1})$, $\beta \in N_1 \cap S(a_1,\dots,a_{n-1})$, it follows that 
$$\alpha - \beta \in \F_p^{\times} \leq (\F_q^{\times})^k.$$   
Hence, $\alpha\beta \in E$ and therefore the bipartite subgraph $H(a_1,\dots,a_{n-1})$ consisting of all edges in $E$ between $N_0 \cap S(a_1,\dots,a_{n-1})$ and $N_1 \cap S(a_1,\dots,a_{n-1})$ is complete. 
In light of Lemma \ref{Lem:PartitionSets}, it follows that that the complete bipartite subgraph $H(a_1,\dots,a_{n-1})$ is balanced. Thus, we can match vertices in $H(a_1,\dots,a_{n-1})$ in a pairwise fashion to obtain a perfect matching between $N_0 \cap S(a_1,\dots,a_{n-1})$ and $N_1 \cap S(a_1,\dots,a_{n-1})$.

Now let $a=(a_1,\dots,a_{n-1}) \in \F_p^{n-1}$ be a multi-index for $a_1,\dots,a_{n-1} \in \F_p$ so that $S(a) = S(a_1,\dots,a_{n-1})$. Set $\mathcal{S} = \{S(a) \mid a \neq 0 \in \F_p^{n-1}\}$ and let $\mathcal{J}$ be a multi-indexing set for $\mathcal{S}$. Noting that $\F_p^{\times} \leq (\F_q^{\times})^k$ implies $\F_p \subseteq \{0\} \cup \{1\} \cup \nabla_{01}$, it follows that
\begin{equation}\label{Eqn:UnionN0}
{\textstyle \bigcup \limits_{a \in \mathcal{J}}} N_0 \cap S(a) =  N_0 \cap {\textstyle \bigcup \limits_{a \in \mathcal{J}}} S(a) = N_0 \cap \left(\F_q \setminus \F_p\right) = N_0
\end{equation}
and similarly
\begin{equation}\label{Eqn:UnionN1}
{\textstyle \bigcup \limits_{a \in \mathcal{J}}} N_1 \cap S(a) = N_1.
\end{equation}
Therefore, since the sets $S(a)=S(a_1,\dots,a_{n-1})$ are disjoint for each distinct choice of $a=(a_1,\dots,a_{n-1}) \in \mathcal{J}$, it follows from \eqref{Eqn:UnionN0} and \eqref{Eqn:UnionN1} that the bipartite subgraph $H$ between $N_0$ and $N_1$ can be partitioned into complete balanced bipartite subgraphs $H(a)=H(a_1,\dots,a_{n-1})$ consisting of all edges in $E$ between $N_0 \cap S(a)$ and $N_1 \cap S(a)$. Thus, since each subgraph $H(a)$ admits a perfect matching, it follows that there is a perfect matching between $N_0$ and $N_1$ and therefore $\Pg(q,k)$ satisfies the Global Matching Condition by symmetry. 
 \end{proof}
 
Due to Theorems \ref{Thm:IntroSmith}, \ref{Thm:kmConnected}, \ref{Thm:PrimeBaseField}, and \ref{Thm:GlobalMatch}, we obtain the following generalization of the unpublished work of \cite{B2} on quadratic Paley graphs $\Pg(q,2)$ of even power order $q=p^{2m}$.

\begin{theorem}\label{Thm:PrimeRicci}
Let $m \geq 1$ and suppose that $\Pg(q,k)=(V,E)$ is a generalized Paley graph of order $q=p^{km}$  where $k$ is prime. Then the condensed Ricci curvature
$$\Bbbk(x,y) = \frac{k}{q-1}(2 + \vert \nabla_{xy} \vert)$$
for any edge $xy \in E$.  
\end{theorem}

\begin{proof}
First note that since $\Pg(q,k)$ is a generalized Paley graph of order $q=p^{km}$ with $k \geq 2$, it follows that $\Pg(q,k)$ is connected by Theorem \ref{Thm:kmConnected}. Moreover, since $k$ is prime it follows that $\F_p^{\times} \leq (\F_q^{\times})^k$ by Theorem \ref{Thm:PrimeBaseField}. Hence, from Theorem \ref{Thm:GlobalMatch} we see that $\Pg(q,k)$ satisfies the Global Matching Condition so that the condensed Ricci curvature
$$\Bbbk(x,y) = \frac{k}{q-1}(2 + \vert \nabla_{xy} \vert)$$
for any edge $xy \in E$ by Theorem \ref{Thm:IntroSmith}.  
\end{proof}

Furthermore, as a consequence of Theorems \ref{Thm:IntroSmith}, \ref{Thm:EquivBaseField}, and \ref{Thm:GlobalMatch}, we obtain the same explicit formula for the condensed Ricci curvature on a large class of connected generalized Paley graphs.

\begin{theorem}\label{Thm:Ricci}
Let $k \geq 2$ and suppose that $\Pg(q,k)=(V,E)$ is a connected generalized Paley graph of order $q=p^n$ where $k \mid \frac{q-1}{p-1}$. Then the condensed Ricci curvature
$$\Bbbk(x,y) = \frac{k}{q-1}(2 + \vert \nabla_{xy} \vert)$$
for any edge $xy \in E$. 
\end{theorem}

\begin{proof}
Since $k \mid \frac{q-1}{p-1}$ it follows that $\F_p^{\times} \leq (\F_q^{\times})^k$ from Theorem \ref{Thm:EquivBaseField}. Hence, $\Pg(q,k)$ satisfies the Global Matching Condition by Theorem \ref{Thm:GlobalMatch}. Therefore, since $\Pg(q,k)$ is connected, it follows from Theorem \ref{Thm:IntroSmith} that the condensed Ricci curvature
$$\Bbbk(x,y) = \frac{k}{q-1}(2 + \vert \nabla_{xy} \vert)$$
for any edge $xy \in E$.
\end{proof}

On the other hand, when a generalized Paley graph $\Pg(q,k)$ with $k \mid \frac{q-1}{p-1}$ is disconnected, our work can be adapted to its connected components to recover the same formula as in Theorem \ref{Thm:Ricci} for the condensed Ricci curvature.

\begin{theorem}\label{Thm:DisconnectedRicci}
Let $k \geq 2$ and suppose that $\Pg(q,k)=(V,E)$ is a disconnected generalized Paley graph order $q=p^n$ where $k \mid \frac{q-1}{p-1}$. Then the condensed Ricci curvature
$$\Bbbk(x,y) = \frac{k}{q-1}(2 + \vert \nabla_{xy} \vert)$$
for any edge $xy \in E$.
\end{theorem}

\begin{proof} From Theorems \ref{Thm:LimP} and \ref{Thm:LimPExtension}, it follows that each of the connected components of $\Pg(q,k)$ is isomorphic to the generalized Paley graph $\Pg(p^a,k')$ with $k'=k(p^a-1)/(q-1) \geq 1$ where $a$ is the smallest proper divisor of $n$ such that $(q-1)/(p^a-1)$ divides $k$. Recalling that the divisibility condition $k \mid \frac{q-1}{p-1}$ is equivalent to $p-1 \mid \frac{q-1}{k}$, it follows that
$$p-1 \mid \frac{q-1}{k} = \frac{p^a-1}{k'} = \vert (\F_{p^a}^{\times})^{k'} \vert$$
and therefore $\F_p^{\times} \leq (\F_{p^a}^{\times})^{k'}$ since $(\F_{p^a}^{\times})^{k'}$ is a finite cyclic group. So for $k' > 1$, applying Theorems \ref{Thm:GlobalMatch} and \ref{Thm:IntroSmith} to the connected components of $\Pg(q,k)$ and noting the fact that $k'=k(p^a-1)/(q-1)$, we see that the condensed Ricci curvature 
$$\Bbbk(x,y) = \frac{k'}{p^a-1}(2 + \vert \nabla_{xy} \vert) = \frac{k}{q-1}(2 + \vert \nabla_{xy} \vert)$$
for any edge $xy \in E$.

In the special case that $k'=1$, we recall that the subgroup of nonzero $1$-st powers in $\F_{p^a}$ is simply the multiplicative group $\F_{p^a}^{\times}$. Therefore, each connected component of the disconnected generalized Paley graph $\Pg(q,k)$ is isomorphic to the complete graph $\Pg(p^a,1) = K_{p^a}$ on $p^a$ vertices. Thus, since the complete graph $K_{p^a}=(V,E)$ on $p^a$ vertices is regular of degree $p^a-1$, it follows for any edge $xy \in E$ that $\nabla_{xy} = V \setminus (\{x\} \cup \{y\})$. Hence, $\vert \nabla_{xy} \vert = p^a-2$ and by Theorem \ref{Thm:CompleteGraphsIntro} we recover the formula 
$$\Bbbk(x,y) = \frac{p^a}{p^a-1} = \frac{1}{p^a-1}(2 + \vert \nabla_{xy} \vert)=\frac{k}{q-1}(2 + \vert \nabla_{xy} \vert)$$
for the condensed Ricci curvature of any edge $xy \in E$.
\end{proof}


\bibliographystyle{plain}
\bibliography{paleygraphs}


\end{document}